\documentclass[reqno]{amsart}
\usepackage{amsfonts}
\usepackage{graphics,graphicx,amssymb}
\usepackage{amsmath,amsthm}
\usepackage{hyperref}
\numberwithin{equation}{section}
\newtheorem{theorem}{Theorem}[section]
\newtheorem{corollary}{Corollary}[section]
\newtheorem{proposition}{Proposition}[section]
\newtheorem{lemma}{Lemma}[section]
\newtheorem{definition}{Definition}[section]

\newfont{\got}{eufm9 scaled 1095}

\newfont{\w}{msbm9 scaled\magstep1}

\begin{document}

\title{Complex connections on conformal K\"ahler manifolds with Norden metric}

\author{Marta Teofilova}
\date{}

\maketitle

\begin{abstract} An eight-parametric family of complex connections on a class complex
manifolds with Norden metric is introduced. The form of the
curvature tensor with respect to each of these connections is
obtained. The conformal group of the considered connections is
studied and some conformal invariants are obtained.

\noindent \emph{2010 Mathematics Subject Classification}: 53C15,
53C50.

\noindent \emph{Keywords}: complex connection, complex manifold,
Norden metric.

\end{abstract}

\maketitle

\section{Introduction}

Almost complex manifolds with Norden metric were first studied by
A.~P.~Norden \cite{N} and are introduced in \cite{Gri-Mek-Dje} as
generalized $B$-manifolds. A classification of these manifolds with
respect to the covariant derivative of the almost complex structure
is obtained in \cite{Gan} and two equivalent classifications are
given in \cite{Gan-Mih,Gan-Gri-Mih2}.

An important problem in the geometry of almost complex manifolds
with Norden metric is the study of linear connections preserving the
almost complex structure or preserving both, the structure and the
metric. The first ones are called almost complex connections, and
the second ones are known as natural connections. A special type of
natural connection is the canonical one. In \cite{Gan-Mih} it is
proved that on an almost complex manifold with Norden metric there
exists a unique canonical connection. The canonical connection
(called also the $B$-connection) and its conformal group on a
conformal K\"{a}hler manifold with Norden metric are studied in
\cite{Gan-Gri-Mih2}.

In \cite{Teo3} we have obtained a two-parametric family of complex
connections on a conformal K\"{a}hler manifold with Norden metric
and have proved that the curvature tensors corresponding to these
connections coincide with the curvature tensor of the canonical
connection.

In the present work we continue our research on complex connections
on complex manifolds with Norden metric by focusing our attention on
the class of the conformal K\"ahler manifolds, i.e. manifolds which
are conformally equivalent to K\"aher manifolds with Norden metric.
We introduce an eight-parametric family of complex connections on
such manifolds and consider their curvature properties. We also
study the conformal group of these connections and obtain some
conformal invariants. In the last section we give an example of a
four-dimensional conformal K\"ahler manifold with Norden metric, on
which the considered complex connections are flat.

\section{Preliminaries}

Let $(M,J,g)$ be a $2n$-dimensional almost complex manifold with
Norden metric, i.e. $J$ is an almost complex structure and $g$ is a
pseudo-Riemannian metric on $M$ such that
\begin{equation}\label{11}
J^{2}x=-x,\qquad g(Jx,Jy)=-g(x,y)
\end{equation}
for all differentiable vector fields $x$, $y$ on $M$, i.e. $x,y\in %
\mathfrak{X}(M)$.

The associated metric $\widetilde{g}$ of $g$ is given by $\widetilde{g}%
(x,y)=g(x,Jy)$ and is a Norden metric, too. Both metrics are
necessarily neutral, i.e. of signature $(n,n)$.

If $\nabla $ is the Levi-Civita connection of the metric $g$, the
fundamental tensor field $F$ of type $(0,3)$ on $M$ is defined by
\begin{equation}\label{F}
F(x,y,z)=g\left((\nabla _{x}J)y,z\right)
\end{equation}
and has the following symmetries
\begin{equation}\label{Fp}
F(x,y,z)=F(x,z,y)=F(x,Jy,Jz).
\end{equation}

Let $\left\{ e_{i}\right\} $ ($i=1,2,\ldots ,2n$) be an arbitrary
basis of $ T_{p}M$ at a point $p$ of $M$. The components of the
inverse matrix of $g$ are denoted by $g^{ij}$ with respect to the
basis $\left\{ e_{i}\right\} $. The Lie 1-forms $\theta $ and
$\theta^{\ast}$ associated with $F$, and the Lie vector $\Omega$,
corresponding to $\theta$, are defined by, respectively
\begin{equation}\label{1-3}
\theta (x)=g^{ij}F(e_{i},e_{j},x), \qquad \theta^{\ast}=\theta \circ
J, \qquad \theta (x)=g(x,\Omega ).
\end{equation}

The Nijenhuis tensor field $N$ for $J$ is given by \cite{Ko-No}
\begin{equation*}
N(x,y)=[Jx,Jy]-[x,y]-J[Jx,y]-J[x,Jy].
\end{equation*}
It is known \cite{N-N} that the almost complex structure is complex
if and only if it is integrable, i.e. iff $N=0$.

A classification of the almost complex manifolds with Norden metric
is introduced in \cite{Gan}, where eight classes of these manifolds
are characterized according to the properties of $F$. The three
basic classes $\mathcal{W}_{i}$ ($i=1,2,3$) are given by

$\bullet$ the class $\mathcal{W}_{1}$:
\begin{equation}\label{w1}
\begin{array}{l}
F(x,y,z)=\frac{1}{2n}\left[ g(x,y)\theta (z)+g(x,Jy)\theta
(Jz)\right. \medskip \\
\quad \qquad \qquad \quad \left. +g(x,z)\theta (y)+g(x,Jz)\theta
(Jy)\right];
\end{array}
\end{equation}

$\bullet$ the class $\mathcal{W}_{2}$ of the \emph{special complex
manifolds with Norden metric}:
\begin{equation}\label{w2}
F(x,y,Jz)+F(y,z,Jx)+F(z,x,Jy)=0,\quad \theta =0
\quad\Leftrightarrow\quad N=0,\quad \theta=0;
\end{equation}

$\bullet$ the class $\mathcal{W}_{3}$ of the \emph{quasi-K\"{a}hler
manifolds with Norden metric}:
\begin{equation}\label{w3}
F(x,y,z)+F(y,z,x)+F(z,x,y)=0.
\end{equation}
The special class $\mathcal{W}_{0}$ of the \emph{K\"{a}hler
manifolds with Norden metric} is characterized by the condition
$F=0$ ($\nabla J=0$) and is contained in each one of the other
classes.

Let $R$ be the curvature tensor of $\nabla $, i.e.
$R(x,y)z=\nabla _{x}\nabla _{y}z-\nabla _{y}\nabla _{x}z-\nabla _{\left[ x,y%
\right] }z$. The corresponding (0,4)-type tensor is defined by
$R(x,y,z,u)=g\left( R(x,y)z,u\right)$.

A tensor $L$ of type (0,4) is called a \emph{curvature-like} tensor
if it has the properties of $R$, i.e.
\begin{equation}\label{L}
\begin{array}{l}
L(x,y,z,u)=-L(y,x,z,u)=-L(x,y,u,z),\medskip\\
L(x,y,z,u)+L(y,z,x,u)+L(z,x,y,u)=0.
\end{array}
\end{equation}

The Ricci tensor $\rho(L)$ and the scalar curvatures $\tau(L)$ and $
\tau^{\ast}(L)$ of $L$ are defined by:
\begin{equation}
\begin{array}{c}
\rho(L)(y,z)=g^{ij}L(e_{i},y,z,e_{j}),\medskip\\
\tau(L)=g^{ij}\rho(L)(e_{i},e_{j}),\quad
\tau^{\ast}(L)=g^{ij}\rho(L) (e_{i},Je_{j}).
\end{array}
\label{Ricci, tao}
\end{equation}

A curvature-like tensor $L$ is called a \emph{K\"{a}hler tensor} if
\begin{equation}\label{Ka}
L(x,y,Jz,Ju) = - L(x,y,z,u).
\end{equation}

Let $S$ be a tensor of type (0,2). We consider the following tensors
\cite{Gan-Gri-Mih2}:
\begin{equation}\label{psi}
\begin{array}{l}
\psi_{1}(S)(x,y,z,u) =g(y,z)S(x,u)-g(x,z)S(y,u) \smallskip\\
\phantom{\psi_{1}(S)(x,y,z,u)}+  g(x,u)S(y,z) - g(y,u)S(x,z), \medskip\\
\psi_{2}(S)(x,y,z,u) =g(y,Jz)S(x,Ju) - g(x,Jz)S(y,Ju) \smallskip\\
 \phantom{\psi_{1}(S)(x,y,z,u)} + g(x,Ju)S(y,Jz) - g(y,Ju)S(x,Jz),  \medskip\\
\pi_{1}=\frac{1}{2}\psi_{1}(g), \qquad
\pi_{2}=\frac{1}{2}\psi_{2}(g),\qquad
\pi_{3}=-\psi_{1}(\widetilde{g})=\psi_{2}(\widetilde{g}).
\end{array}
\end{equation}
The tensor $\psi_{1}(S)$ is curvature-like if $S$ is symmetric, and
the tensor $\psi_{2}(S)$ is curvature-like if $S$ is symmetric and
hybrid with respect to $J$, i.e. $S(x,Jy)=S(y,Jx)$. In the last case
the tensor $\{\psi_1 - \psi_2\}(S)$ is K\"ahlerian. The tensors
$\pi_{1} - \pi_{2}$ and $\pi_{3}$ are also K\"{a}hlerian.

The usual conformal transformation of the Norden metric $g$
(conformal transformation of type I \cite{Gan-Gri-Mih2}) is defined
by
\begin{equation}\label{conf}
\overline{g}=e^{2u}g,
\end{equation}
where $u$ is a pluriharmonic function, i.e. the 1-form $du\circ J$
is closed.

A $\mathcal{W}_1$-manifold with closed 1-forms $\theta$ and
$\theta^{\ast}$ (i.e. $\mathrm{d}\theta=\mathrm{d}\theta^\ast=0$) is
called a \emph{conformal K\"ahler manifold with Norden metric}.
Necessary and sufficient conditions for a $\mathcal{W}_1$-manifold
to be conformal K\"ahlerian are:
\begin{equation}\label{cK}
(\nabla_x\theta)y=(\nabla_y\theta)x,\qquad
(\nabla_x\theta)Jy=(\nabla_y\theta)Jx.
\end{equation}
The subclass of these manifolds is denoted by
$\mathcal{W}_{1}^{\hspace{0.01in}0}$.

It is proved \cite{Gan-Gri-Mih2} that a
$\mathcal{W}_{1}^{\hspace{0.01in}0}$-manifold is conformally
equivalent to a K\"ahler manifold with Norden metric by the
transformation (\ref{conf}).

It is known that on a pseudo-Riemannian manifold $M$ ($\dim M=2n
\geq 4$) the conformal invariant Weyl tensor has the form
\begin{equation}\label{Weyl}
W(R)=R-\frac{1}{2(n-1)}\big
\{\psi_{1}(\rho)-\frac{\tau}{2n-1}\pi_{1}\big \}.
\end{equation}

Let $L$ be a K\"ahler curvature-like tensor on an almost complex
manifold with Norden metric $(M,J,g)$, $\dim M=2n\geq 6$. Then the
Bochner tensor $B(L)$ for $L$ is defined by \cite{Gan-Gri-Mih2}:
\begin{equation}\label{Bochner}
\begin{array}{l}
B(L)= L -
\frac{1}{2(n-2)}\big\{\psi_{1}-\psi_{2}\big\}\big(\rho(L)\big)\medskip\\
\phantom{B(L)}+
\frac{1}{4(n-1)(n-2)}\big\{\tau(L)\big(\pi_{1}-\pi_{2}\big) +
\tau^{\ast}(L)\pi_{3}\big\}.
\end{array}
\end{equation}

\section{Complex Connections on $\mathcal{W}_1$-manifolds}

\begin{definition}[\cite{Ko-No}]\label{def-complex} \emph{A linear connection $\nabla^{\prime}$ on
an almost complex manifold $(M,J)$ is said to be} almost complex
\emph{if $\nabla^{\prime}J=0$.}
\end{definition}

We introduce an eight-parametric family of complex connections in
the following
\begin{theorem} On a $\mathcal{W}_1$-manifold with Norden metric
there exists an eight-parametric family of complex connections
$\nabla^{\prime}$ defined by

\begin{equation}\label{2-1}
\begin{array}{l}
\nabla_x^{\prime}y = \nabla_x y + Q(x,y),
\end{array}
\end{equation}
where the deformation tensor $Q(x,y)$ is given by
\begin{equation}\label{2-2}
\begin{array}{l}
Q(x,y)= \frac{1}{2n}\left[\theta(Jy)x-g(x,y)J\Omega\right]\medskip\\
\hspace{0.065in}\phantom{Q(x,y)}+
\frac{1}{n}\left\{\lambda_1\theta(x)y+\lambda_2\theta(x)Jy
+\lambda_3\theta(Jx)y +\lambda_4\theta(Jx)Jy\right.\medskip\\
\phantom{Q(x,y)}\hspace{0.06in}+\lambda_5\left[\theta(y)x-\theta(Jy)Jx\right]
+
\lambda_6\left[\theta(y)Jx+\theta(Jy)x\right]\medskip\\
\left.\hspace{0.06in}\phantom{Q(x,y)}+\lambda_7\left[g(x,y)\Omega-g(x,Jy)J\Omega\right]+\lambda_8\left[g(x,Jy)\Omega+g(x,y)J\Omega\right]\right\},
\end{array}
\end{equation}
$\lambda_i\in \mathbb{R}$, $i=1,2,...,8$.
\end{theorem}

\begin{proof} By (\ref{w1}), (\ref{2-1}) and (\ref{2-2}) we verify that
$(\nabla^{\prime}_xJ)y = \nabla^{\prime}_xJy -
J\nabla^{\prime}_xy=0$, and hence the connections $\nabla^{\prime}$
are complex for any $\lambda_i\in \mathbb{R}$, $i=1,2,...,8$.
\end{proof}

Let us remark that the two-parametric family of complex connections
obtained for $\lambda_1=\lambda_4$, $\lambda_3=-\lambda_2$,
$\lambda_5=\lambda_7=0$, $\lambda_8=-\lambda_6=\frac{1}{4}$, is
studied in \cite{Teo3}.

Let $T^{\prime}$ be the torsion tensor of $\nabla^{\prime}$, i.e.
$T^{\prime}(x,y)=\nabla^{\prime}_xy - \nabla^{\prime}_xy - [x,y]$.
Taking into account that the Levi-Civita connection $\nabla$ is
symmetric, we have $T^{\prime}(x,y)=Q(x,y)-Q(y,x)$. Then by
(\ref{2-2}) we obtain
\begin{equation}\label{T}
\begin{array}{l}
T^{\prime}(x,y)=\frac{1}{n}\left\{\left(\lambda_1-\lambda_5\right)\left[\theta(x)y-\theta(y)x\right]
+\left(\lambda_2-\lambda_6\right)\left[\theta(x)Jy-\theta(y)Jx\right]\right.\medskip\\
\qquad\quad\hspace{0.02in}\left.+\left(\lambda_3-\lambda_6-\frac{1}{2}\right)\left[\theta(Jx)y-\theta(Jy)x\right]
+\left(\lambda_4+\lambda_5\right)\left[\theta(Jx)Jy-\theta(Jy)Jx\right]
\right\}.
\end{array}
\end{equation}

It is easy to verify the following
\begin{equation}
\underset{x,y,z}{\mathfrak{S}}T^{\prime}(x,y,z)=\underset{x,y,z}{\mathfrak{S}}T^{\prime}(Jx,Jy,z)=\underset{x,y,z}{\mathfrak{S}}T^{\prime}(x,y,Jz)=0,
\end{equation}
where $\mathfrak{S}$ is the cyclic sum by the arguments $x,y,z$.

Next, we obtain necessary and sufficient conditions for the complex
connections $\nabla^{\prime}$ to be symmetric (i.e. $T^{\prime}=0$).
\begin{theorem}
The complex connections $\nabla^{\prime}$ defined by (\ref{2-1}) and
(\ref{2-2}) are symmetric on a $\mathcal{W}_1$-manifold with Norden
metric if and only if $\lambda_1=-\lambda_4=\lambda_5$,
$\lambda_2=\lambda_3-\frac{1}{2}=\lambda_6$.
\end{theorem}
Then, by putting $\lambda_1=-\lambda_4=\lambda_5=\mu_1$,
$\lambda_2=\lambda_6=\lambda_3-\frac{1}{2}=\mu_2$,
$\lambda_7=\mu_3$, $\lambda_8=\mu_4$ in (\ref{2-2}) we obtain a
four-parametric family of complex symmetric connections
$\nabla^{\prime\prime}$ on a $\mathcal{W}_1$-manifold which are
defined by
\begin{equation}\label{sym}
\begin{array}{l}
\nabla^{\prime\prime}_x y= \nabla_x y +
\frac{1}{2n}\left[\theta(Jx)y+\theta(Jy)x-g(x,y)J\Omega\right]\medskip\\
\phantom{\nabla^{\prime}_x
y}+\frac{1}{n}\left\{\mu_1\left[\theta(x)y+\theta(y)x-\theta(Jx)Jy-\theta(Jy)Jx\right]\right.\medskip\\
\phantom{\nabla^{\prime}_x
y}+\mu_2\left[\theta(Jx)y+\theta(Jy)x+\theta(x)Jy+\theta(y)Jx\right]\medskip\\
\phantom{\nabla^{\prime}_x
y}+\left.\mu_3\left[g(x,y)\Omega-g(x,Jy)J\Omega\right]+\mu_4\left[g(x,Jy)\Omega+g(x,y)J\Omega\right]\right\}.
\end{array}
\end{equation}
The well-known Yano connection \cite{Ya1,Ya2} on a
$\mathcal{W}_1$-manifold is obtained from (\ref{sym}) for
$\mu_1=\mu_3=0$, $\mu_4=-\mu_2=\frac{1}{4}$.

\begin{definition}[\cite{Gan-Mih}]\label{def-nat} \emph{A linear connection $\nabla^{\prime}$ on
an almost complex manifold with Norden metric $(M,J,g)$ is said to
be} natural \emph{if $\nabla^{\prime}J=\nabla^{\prime}g=0$ (or
equivalently, $\nabla^{\prime}g=\nabla^{\prime}\widetilde{g}=0$).}
\end{definition}

From (\ref{2-1}) and (\ref{2-2}) we compute the covariant
derivatives of $g$ and $\widetilde{g}$ with respect to the complex
connections $\nabla^{\prime}$ as follows
\begin{equation}\label{2-5}
\begin{array}{l}
\left(\nabla^{\prime}_x g\right)(y,z)=-Q(x,y,z)-Q(x,z,y)\medskip\\
=-\frac{1}{n}\left\{
2\left[\lambda_1\theta(x)g(y,z)+\lambda_2\theta(x)g(y,Jz)+\lambda_3\theta(Jx)g(y,z)\right.\right.\medskip \\
\left.+\lambda_4\theta(Jx)g(y,Jz)\right]+(\lambda_5+\lambda_7)\left[\theta(y)g(x,z)+\theta(z)g(x,y)\right.\medskip\\
\left.-\theta(Jy)g(x,Jz)-\theta(Jz)g(x,Jy)\right]+(\lambda_6+\lambda_8)\left[\theta(y)g(x,Jz)\right.\medskip\\
\left.+\theta(z)g(x,Jy)+\theta(Jy)g(x,z)+\theta(Jz)g(x,y)\right]\left.\right\},\medskip\\
\left(\nabla^{\prime}_x
\widetilde{g}\right)(y,z)=-Q(x,y,Jz)-Q(x,Jz,y).
\end{array}
\end{equation}
Then, by (\ref{2-5}) we get the following

\begin{theorem}
The complex connections $\nabla^{\prime}$ defined by (\ref{2-1}) and
(\ref{2-2}) are natural on a $\mathcal{W}_1$-manifold if and only if
$\lambda_1=\lambda_2=\lambda_3=\lambda_4=0$, $\lambda_7=-\lambda_5$,
$\lambda_8=-\lambda_6$.
\end{theorem}

If we put $\lambda_8=-\lambda_6=s$, $\lambda_7=-\lambda_5=t$,
$\lambda_i=0$, $i=1,2,3,4$, in (\ref{2-2}), we obtain a
two-parametric family of natural connections
$\nabla^{\prime\prime\prime}$ defined by
\begin{equation}\label{nabla-n}
\begin{array}{l}
\nabla^{\prime\prime\prime}_x y = \nabla_x y
+\frac{1-2s}{2n}\left[\theta(Jy)x-g(x,y)J\Omega\right]
+\frac{1}{n}\left\{s\left[g(x,Jy)\Omega -
\theta(y)Jx\right]\right.\medskip\\
\phantom{\nabla^{\prime}_x y}\left.+t\left[g(x,y)\Omega -
g(x,Jy)J\Omega - \theta(y)x+\theta(Jy)Jx\right]\right\}.
\end{array}
\end{equation}
The well-known canonical connection \cite{Gan-Mih} (or
$B$-connection \cite{Gan-Gri-Mih2}) on a $\mathcal{W}_1$-manifold
with Norden metric is obtained from (\ref{nabla-n}) for
$s=\frac{1}{4}$, $t=0$.

We give a summery of the obtained results in the following table
\begin{center}
\begin{tabular}{|l|c|c|}
  \hline
  Connection type & Symbol  & Parameters \\ \hline
  $\begin{array}{l}\text{Complex}\end{array}$ & $\nabla^{\prime}$ & $\lambda_i\in\mathbb{R}$, $i=1,2,...,8$. \\
  \hline
  $\begin{array}{l}\text{Complex}\smallskip\\ \text{symmetric}\end{array}$ & $\nabla^{\prime\prime}$ & $\begin{array}{c}\mu_i,\hspace{0.03in} i=1,2,3,4, \smallskip\\
  \mu_1=\lambda_1=-\lambda_4=\lambda_5,\hspace{0.03in}
\mu_2=\lambda_2=\lambda_6=\lambda_3-\frac{1}{2},\smallskip\\
\mu_3=\lambda_7,\hspace{0.03in} \mu_4=\lambda_8\end{array}$ \\
\hline
  $\begin{array}{l}\text{Natural}\end{array}$ & $\nabla^{\prime\prime\prime}$ & $\begin{array}{c}s,t,\smallskip\\
  s=\lambda_8=-\lambda_6, \hspace{0.03in} t = \lambda_7=-\lambda_5,\smallskip\\
  \lambda_i=0,\hspace{0.03in} i =1,2,3,4.\end{array}$ \\
  \hline
\end{tabular}

\end{center}

\medskip

Our next aim is to study the curvature properties of the complex
connections $\nabla^{\prime}$. Let us first consider the natural
connection $\nabla^0$ obtained from (\ref{nabla-n}) for $s=t=0$,
i.e.
\begin{equation}\label{nabla-0}
\nabla^{0}_x y = \nabla_x y + \frac{1}{2n}\left[\theta(Jy)x -
g(x,y)J\Omega\right].
\end{equation}
This connection is a semi-symmetric metric connection, i.e. a
connection of the form $\nabla_x y + \omega(y)x - g(x,y)U$, where
$\omega$ is a 1-form and $U$ is the corresponding vector of
$\omega$, i.e. $\omega(x)=g(x,U)$. Semi-symmetric metric connections
are introduced in \cite{Ha} and studied in \cite{Im,Ya3}. The form
of the curvature tensor of an arbitrary connection of this type is
obtained in \cite{Ya3}. The geometry of such connections on almost
complex manifolds with Norden metric is considered in \cite{Si}.

Let us denote by $R^0$ the curvature tensor of $\nabla^0$, i.e.
$R^0(x,y)z=\nabla^0_x\nabla^0_y z - \nabla^0_y\nabla^0_x z -
\nabla^0_{[x,y]}z$. The corresponding tensor of type (0,4) is
defined by $R^0(x,y,z,u)=g(R^0(x,y,)z,u)$. According to \cite{Ya3}
it is valid

\begin{proposition}\label{2}
On a $\mathcal{W}_1$-manifold with closed 1-form $\theta^{\ast}$ the
K\"ahler curvature tensor $R^0$ of $\nabla^0$ defined by
(\ref{nabla-0}) has the form
\begin{equation}\label{R0}
R^0=R - \frac{1}{2n}\psi_1(P),
\end{equation}
where
\begin{equation}\label{R01}
P(x,y)=\left(\nabla_x \theta\right)Jy +
\frac{1}{2n}\theta(x)\theta(y)+\frac{\theta(\Omega)}{4n}g(x,y)+\frac{\theta(J\Omega)}{2n}g(x,Jy).
\end{equation}
\end{proposition}

Since the Weyl tensor $W(\psi_1(S))=0$ (where $S$ is a symmetric
(0,2)-tensor), from (\ref{R0}) and (\ref{R01}) we conclude that
\begin{equation}\label{WW}
W(R^0)=W(R).
\end{equation}
Thus, the last equality implies
\begin{proposition}\label{thW}
Let $(M,J,g)$ be a $\mathcal{W}_1$-manifold with closed 1-form
$\theta^{\ast}$, and $\nabla^0$ be the natural connection defined by
(\ref{nabla-0}). Then, the Weyl tensor is invariant by the
transformation $\nabla \rightarrow \nabla^0$.
\end{proposition}

Further in this section we study the curvature properties of the
complex connections $\nabla^{\prime}$ defined by (\ref{2-1}) and
(\ref{2-2}). Let us denote by $R^{\prime}$ the curvature tensors
corresponding to these connections.

If a linear connection $\nabla^\prime$ and the Levi-Civita
connection $\nabla$ are related by an equation of the form
(\ref{2-1}), then, because of $\nabla g=0$, their curvature tensors
of type (0,4) satisfy
\begin{equation}\label{33}
\begin{array}{l}
g(R^\prime(x,y)z,u) = R(x,y,z,u) + (\nabla_x Q)(y,z,u) - (\nabla_y
Q)(x,z,u)\medskip\\
\phantom{g(R^\prime(x,y)z,u)} + Q(x,Q(y,z),u) - Q(y,Q(x,z),u),
\end{array}
\end{equation}
where $Q(x,y,z) = g(Q(x,y),z)$.

Then, by (\ref{2-1}), (\ref{2-2}), (\ref{R0}), (\ref{R01}),
(\ref{33}) we obtain the relation between $R^{\prime}$ and $R^0$ as
follows
\begin{equation}\label{h1}
\begin{array}{l}
R^{\prime}(x,y,z,u) = R^{0}(x,y,z,u) +g(y,z)A_1(x,u) -
g(x,z)A_1(y,u)\medskip\\ + g(x,u)A_2(y,z) - g(y,u)A_2(x,z)
-g(y,Jz)A_1(x,Ju)\medskip\\ +
g(x,Jz)A_1(y,Ju)-g(x,Ju)A_2(y,Jz)+g(y,Ju)A_2(x,Jz)\medskip\\
+\left[\frac{\lambda_5\lambda_7 -
\lambda_6\lambda_8}{n^2}\theta(\Omega) + \frac{\lambda_7 - \lambda_5
+ 2(\lambda_5\lambda_8+\lambda_6\lambda_7)}{2n^2}\theta(J\Omega)
\right]\{\pi_1-\pi_2\}(x,y,z,u)\medskip\\
-\left[\frac{\lambda_5\lambda_8 +
\lambda_6\lambda_7}{n^2}\theta(\Omega) - \frac{\lambda_6 - \lambda_8
+
2(\lambda_5\lambda_7-\lambda_6\lambda_8)}{2n^2}\theta(J\Omega)\right]\pi_3(x,y,z,u),
\end{array}
\end{equation}
where
\begin{equation}\label{h2}
\begin{array}{l}
A_1(x,y) =
\frac{\lambda_7}{n}\left\{\left(\nabla_x\theta\right)y+\frac{\lambda_7}{n}[\theta(x)\theta(y)-\theta(Jx)\theta(Jy)]\right\}\medskip\\
\phantom{A_1(x,y)}+\frac{\lambda_8}{n}\left\{\left(\nabla_x\theta\right)Jy
+
\frac{1-2\lambda_8}{2n}[\theta(x)\theta(y)-\theta(Jx)\theta(Jy)]\right\}\medskip\\
\phantom{A_1(x,y)}+\frac{\lambda_7(4\lambda_8-1)}{2n^2}[\theta(x)\theta(Jy)+\theta(Jx)\theta(y)],\bigskip\\
A_2(x,y)=-\frac{\lambda_5}{n}\left\{\left(\nabla_x\theta\right)y -
\frac{\lambda_5}{n}[\theta(x)\theta(y)-\theta(Jx)\theta(Jy)]\right\}\medskip\\
\phantom{A_2(x,y)}-\frac{\lambda_6}{n}\left\{\left(\nabla_x\theta\right)Jy+\frac{1+2\lambda_6}{2n}[\theta(x)\theta(y)-\theta(Jx)\theta(Jy)]\right\}\medskip\\
\phantom{A_2(x,y)}+\frac{\lambda_5(4\lambda_6+1)}{2n^2}[\theta(x)\theta(Jy)+\theta(Jx)\theta(y)].
\end{array}
\end{equation}
We are interested in necessary and sufficient conditions for
$R^{\prime}$ to be a K\"ahler curvature-like tensor, i.e. to satisfy
(\ref{L}) and (\ref{Ka}). From (\ref{psi}), (\ref{h1}) and
(\ref{h2}) it follows that $R^{\prime}$ is K\"ahlerian if and only
if $A_1(x,y)=A_2(x,y)$. Hence we obtain

\begin{theorem}\label{1}
Let $(M,J,g)$ be a conformal K\"ahler manifold with Norden metric,
and $\nabla^{\prime}$ be the complex connection defined by
(\ref{2-1}) and (\ref{2-2}). Then, $R^{\prime}$ is a K\"ahler
curvature-like tensor if and only if $\lambda_7=-\lambda_5$ and
$\lambda_8=-\lambda_6$. In this case, from (\ref{2-1}) and
(\ref{2-2}) we obtain a six-parametric family of complex connections
$\nabla^{\prime}$ whose curvature tensors $R^{\prime}$ have the form
\begin{equation}
\begin{array}{l}\label{Rpr}
R^{\prime} = R^0 + \frac{\lambda_7}{n}\left\{\psi_1 -
\psi_2\right\}(S_1) +
\frac{\lambda_8}{n}\left\{\psi_1-\psi_2\right\}(S_2)
\medskip\\ \phantom{R^{\prime}}+
\frac{\lambda_7(4\lambda_8-1)}{2n^2}\left\{\psi_1-\psi_2\right\}(S_3)
+
\frac{\lambda_7(1-2\lambda_8)\theta(J\Omega)}{n^2}\left\{\pi_1-\pi_2\right\}\medskip\\
\phantom{R^{\prime}}+\frac{2\lambda_7\lambda_8\theta(\Omega)}{n^2}\pi_3,
\end{array}
\end{equation}
where $R^0$ is given by (\ref{R0}), (\ref{R01}), and
\begin{equation}\label{111}
\begin{array}{l}
S_1(x,y) = \left(\nabla_x\theta\right)y +
\frac{\lambda_7}{n}[\theta(x)\theta(y)-\theta(Jx)\theta(Jy)] -
\frac{\lambda_7\theta(\Omega)}{2n}g(x,y) \medskip\\
\phantom{S_1(x,y)} +
\frac{\lambda_7\theta(J\Omega)}{2n}g(x,Jy),\medskip\\
S_2(x,y) = \left(\nabla_x\theta\right)Jy +
\frac{1-2\lambda_8}{2n}[\theta(x)\theta(y)-\theta(Jx)\theta(Jy)]\medskip\\
\phantom{S_2(x,y)}+\frac{\lambda_8\theta(\Omega)}{2n}g(x,y) +
\frac{(1-\lambda_8)\theta(J\Omega)}{2n}g(x,Jy),\medskip\\
S_3(x,y) = \theta(x)\theta(Jy) + \theta(Jx)\theta(y).
\end{array}\end{equation}
\end{theorem}

From (\ref{Rpr}), (\ref{111}) and Theorem \ref{1} we get

\begin{corollary}\label{cor}
Let $(M,J,g)$ be a conformal K\"ahler manifold with Norden metric
and $\nabla^{\prime}$ be the eight-parametric family of complex
connections defined by (\ref{2-1}) and (\ref{2-2}). Then
$R^{\prime}=R^0$ if and only if $\lambda_i=0$ for $i=5,6,7,8$.
\end{corollary}

Let us remark that by putting $\lambda_i=0$ for $i=1,2,5,6,7,8$ in
(\ref{2-2}) we obtain a two-parametric family of complex connections
whose K\"ahler curvature tensors coincides with $R^0$ on a
$\mathcal{W}_1$-manifold with closed 1-form $\theta^{\ast}$.

Theorem \ref{thW} and Corollary \ref{cor} imply

\begin{corollary}
On a conformal K\"ahler manifold with Norden metric the Weyl tensor
is invariant by the transformation of the Levi-Civita connection in
any of the complex connection $\nabla^{\prime}$ defined by
(\ref{2-1}) and (\ref{2-2}) for $\lambda_i=0$, $i=5,6,7,8$.
\end{corollary}

Since for the Bochner tensor of $\{\psi_1-\psi_2\}(S)$ it is valid
$B\left(\{\psi_1-\psi_2\}(S)\right)=0$, where $S$ is symmetric and
hybrid with respect to $J$, from Theorem \ref{1} and (\ref{psi}) it
follows
\begin{equation}\label{BR0}
B(R^{\prime}) = B(R^0).
\end{equation}
By this way we proved the truthfulness of the following
\begin{theorem}\label{thB}
Let $(M,J,g)$ be a conformal K\"ahler manifold with Norden metric,
$R^{\prime}$ be the curvature tensor of $\nabla^{\prime}$ defined by
(\ref{2-1}) and (\ref{2-2}) for $\lambda_7=-\lambda_5$,
$\lambda_8=-\lambda_6$, and $R^0$ be the curvature tensor of
$\nabla^0$ given by (\ref{nabla-0}). Then the Bochner tensor is
invariant by the transformations
$\nabla^0\rightarrow\nabla^{\prime}$.
\end{theorem}

\section{Conformal transformations of complex connections}

In this section we study usual conformal transformations of the
complex connections $\nabla^{\prime}$ defined in the previous
section.

Let $(M,J,g)$ and $(M,J,\bar{g})$ be conformally equivalent almost
complex manifolds with Norden metric by the transformation
(\ref{conf}). It is known that the Levi-Civita connections $\nabla$
and $\overline{\nabla}$ of $g$ and $\overline{g}$, respectively, are
related as follows
\begin{equation}\label{con-trans1}
\overline{\nabla}_{x}y = \nabla_{x}y + \sigma(x)y + \sigma(y)x -
g(x,y)\Theta,
\end{equation}
where $\sigma(x)=du(x)$ and $\Theta=\textrm{grad}\hspace{0.02in}
\sigma$, i.e. $\sigma(x)=g(x,\Theta)$. Let $\overline{\theta}$ be
the Lie 1-form of $(M,J,\overline{g})$. Then by (\ref{1-3}) and
(\ref{con-trans1}) we obtain
\begin{equation}\label{theta-bar}
\overline{\theta} = \theta + 2n\big(\sigma\circ J\big),\qquad\quad
\overline{\Omega}=e^{-2u}\big(\Omega + 2nJ\Theta\big).
\end{equation}

It is valid the following
\begin{lemma}\label{lemma1}
Let $(M,J,g)$ be an almost complex manifold with Norden metric and
$(M,J,\overline{g})$ be its conformally equivalent manifold by the
transformation (\ref{conf}). Then, the curvature tensors $R$ and
$\overline{R}$ of $\nabla$ and $\overline{\nabla}$, respectively,
are related as follows
\begin{equation}\label{Rbar}
\overline{R}=e^{2u}\big\{R-\psi_{1}\big(V\big) -
\pi_{1}\sigma\big(\Theta\big)\big\},
\end{equation}
where $V(x,y)=\big(\nabla_{x}\sigma\big)y - \sigma(x)\sigma(y)$.
\end{lemma}

Let us first study the conformal group of the natural connection
$\nabla^0$ given by (\ref{nabla-0}). Equalities (\ref{nabla-0}) and
(\ref{con-trans1}) imply that its conformal group is defined
analytically by
\begin{equation}\label{nabla0-bar}
\overline{\nabla}^{\hspace{0.02in}0}_{x}y = \nabla^{0}_{x}y +
\sigma(x)y.
\end{equation}
It is known that if two linear connections are related by an
equation of the form (\ref{nabla0-bar}), where $\sigma$ is a 1-form,
then the curvature tensors of these connections coincide iff
$\sigma$ is closed. Hence, it is valid
\begin{theorem}\label{thR0}
Let $(M,J,g)$ be a $\mathcal{W}_1$-manifold with closed 1-form
$\theta^{\ast}$. Then the curvature tensor $R^0$ of $\nabla^0$ is
conformally invariant, i.e.
\begin{equation}
\overline{R}^0 = e^{2u}R^0.
\end{equation}
\end{theorem}

Further in this section let $(M,J,g)$ be a conformal K\"ahler
manifold with Norden metric. Then $(M,J,\overline{g})$ is a K\"ahler
manifold and thus $\overline{\theta}=0$. From (\ref{theta-bar}) it
follows $\sigma =\frac{1}{2n}(\theta \circ J)$. Then, from
(\ref{2-1}) and (\ref{2-2}) we get
$\overline{\nabla}^{\hspace{0.02in}\prime}=\overline{\nabla}$ and
hence $\overline{R}^{\prime}=\overline{R}$ for all
$\lambda_i\in\mathbb{R}$, $i=1,2,...,8$. In particular,
$\overline{R}^\prime=\overline{R}^0$. Then, Theorems \ref{thB} and
(\ref{BR0}) imply

\begin{theorem}
On a conformal K\"ahler manifold with Norden metric the Bochner
curvature tensor of the complex connections $\nabla^{\prime}$
defined by (\ref{2-1}) and (\ref{2-2}) with the conditions
$\lambda_7=-\lambda_5$ and $\lambda_8=-\lambda_6$ is conformally
invariant by the transformation (\ref{conf}), i.e.
\begin{equation}
B(\overline{R}^{\prime})=e^{2u}B(R^{\prime}).
\end{equation}
\end{theorem}

Let us remark that the conformal invariancy of the Bochner tensor of
the canonical connection on a conformal K\"ahler manifold with
Norden metric is proved in \cite{Gan-Gri-Mih2}.

From Theorem \ref{thR0} and Corollary \ref{cor} we obtain

\begin{corollary}
Let $(M,J,g)$ be a conformal K\"ahler manifold with Norden metric
and $\nabla^{\prime}$ be a complex connection defined by (\ref{2-1})
and (\ref{2-2}). If $\lambda_i=0$ for $i=5,6,7,8$, then the
curvature tensor of $\nabla^{\prime}$ is conformally invariant by
the transformation (\ref{conf}).
\end{corollary}

\section{An Example}

Let $G$ be a real connected four-dimensional Lie group, and
$\mathfrak{g}$ be its corresponding Lie algebra. If
$\{e_1,e_2,e_3,e_4\}$ is a basis of $\mathfrak{g}$, we equip $G$
with a left-invariant almost complex structure $J$ by
\begin{equation}\label{J1}
Je_1 = e_3,\qquad Je_2 = e_4,\qquad Je_3=-e_1,\qquad Je_4=-e_2.
\end{equation}
We also define a left-invariant pseudo-Riemannian metric $g$ on $G$
by
\begin{equation}\label{g1}
\begin{array}{l}
g(e_1,e_1)=g(e_2,e_2)=-g(e_3,e_3)=-g(e_4,e_4)=1,\medskip\\
g(e_i,e_j)=0,\quad i\neq j,\quad i,j=1,2,3,4.
\end{array}
\end{equation}
Then, because of (\ref{11}), (\ref{J1}) and (\ref{g1}), $(G,J,g)$ is
an almost complex manifold with Norden metric.

Further, let the Lie algebra $\mathfrak{g}$ be defined by the
following commutator relations
\begin{equation}\label{lie1}
\begin{array}{l}
[e_1,e_2]=[e_3,e_4]=0,\medskip\\
\lbrack e_1,e_4]=[e_2,e_3]=\lambda (e_1 + e_4) + \mu (e_2 - e_3),\medskip\\
\lbrack e_1,e_3]=-[e_2,e_4] = \mu (e_1 + e_4) - \lambda (e_2 - e_3),
\end{array}
\end{equation}
where $\lambda,\mu\in\mathbb{R}$.

The well-known Koszul's formula for the Levi-Civita connection of
$g$ on $G$, i.e. the equality
\begin{equation}\label{K}
2g(\nabla_{e_i} e_j,e_k) = g([e_i,e_j],e_k) +
g([e_k,e_i],e_j)+g([e_k,e_j],e_i),
\end{equation}
and (\ref{g1}) imply the following essential non-zero components of
the Levi-Civita connection:
\begin{equation}\label{nabla1}
\begin{array}{ll}
\nabla_{e_1}e_1 = \nabla_{e_2}e_2 = \mu e_3 + \lambda e_4,\quad &
\nabla_{e_3}e_3 = \nabla_{e_4}e_4 = -\lambda e_1 + \mu
e_2,\medskip\\
\nabla_{e_1}e_3 = \mu (e_1 + e_4),\quad & \nabla_{e_1}e_4 = \lambda
e_1 - \mu e_3,\medskip\\
\nabla_{e_2}e_3 = \mu e_1 + \lambda e_4,\quad & \nabla_{e_2}e_4 =
\lambda (e_2 - e_3).
\end{array}
\end{equation}

Then, by (\ref{F}), (\ref{Fp}) and (\ref{nabla1}) we compute the
following essential non-zero components $F_{ijk}=F(e_i,e_j,e_k)$ of
$F$:
\begin{equation}\label{Fijk}
\begin{array}{l}
F_{111} = F_{422} = 2\mu,\quad F_{222}=-F_{311} =
2\lambda,\medskip\\
F_{112} = -F_{214} = F_{314} = -F_{412} = \lambda,\quad
F_{212}=-F_{114}=F_{312}=-F_{414}=\mu.
\end{array}
\end{equation}

Having in mind (\ref{1-3}) and (\ref{Fijk}), the components
$\theta_i=\theta(e_i)$ and $\theta_i^\ast=\theta^\ast(e_i)$ of the
1-forms $\theta$ and $\theta^\ast$, respectively, are:
\begin{equation}\label{theta}
\theta_2=\theta_3=\theta^\ast_1=-\theta^\ast_4 =
4\lambda,\qquad\theta_1=-\theta_4=-\theta^\ast_2=-\theta^\ast_3=4\mu.
\end{equation}

By (\ref{1-3}) and (\ref{theta}) we compute
\begin{equation}\label{22}
\begin{array}{c}
\Omega = 4\mu(e_1+e_4) + 4\lambda(e_2 -e_3),\qquad J\Omega =
4\lambda(e_1 + e_4) - 4\mu(e_2 - e_3),\medskip\\ \theta (\Omega) =
\theta(J\Omega) = 0.
\end{array}
\end{equation}

By the characteristic condition (\ref{w1}) and equalities
(\ref{Fijk}), (\ref{theta}) we prove that the manifold $(G,J,g)$
with Lie algebra $\mathfrak{g}$ defined by (\ref{lie1}) belongs to
the basic class $\mathcal{W}_1$. Moreover, by (\ref{nabla1}) and
(\ref{theta}) it follows that the conditions (\ref{cK}) hold and
thus
\begin{proposition}
The manifold $(G,J,g)$ defined by (\ref{J1}), (\ref{g1}) and
(\ref{lie1}) is a conformal K\"ahler manifold with Norden metric.
\end{proposition}

According to (\ref{nabla-0}), (\ref{g1}), (\ref{nabla1}) and
(\ref{theta}) the components of the natural connection $\nabla^0$
are given by
\begin{equation}\label{nabla-01}
\begin{array}{ll}
\nabla_{e_1}^0 e_1 = - \nabla_{e_4}^0 e_1 = \mu e_2,\qquad &
\nabla_{e_2}^0 e_1 = \nabla_{e_3}^0 e_1 = \lambda e_2,\medskip\\
\nabla_{e_1}^0 e_2 = -\nabla_{e_4}^0 e_2 = -\mu e_1,\qquad &
\nabla_{e_2}^0 e_2 = \nabla_{e_3}^0 e_2 = - \lambda e_1,\medskip\\
\nabla_{e_1}^0 e_3 =-\nabla_{e_4}^0 e_3 = \mu e_4,\qquad &
\nabla_{e_2}^0 e_3 = \nabla_{e_3}^0 e_3 = \lambda e_4,\medskip\\
\nabla_{e_1}^0 e_4 = -\nabla_{e_4}^0 e_4 = -\mu e_3,\qquad &
\nabla_{e_2}^0 e_4 = \nabla_{e_4}^0 e_4 = - \lambda e_3.
\end{array}
\end{equation}
By (\ref{nabla-01}) we obtain $R^0=0$. Then, by (\ref{R0}) and
(\ref{R01}) the curvature tensor $R$ of $(G,J,g)$ has the form
\begin{equation}
R = \frac{1}{4}\psi_1(A),\qquad A(x,y) = (\nabla_x\theta)Jy +
\frac{1}{4}\theta(x)\theta(y).
\end{equation}

Moreover, having in mind (\ref{111}), (\ref{nabla1}), (\ref{theta})
and (\ref{22}), we compute $S_1=S_2=S_3=0$. Hence, for the tensors
$R^\prime$ of the complex connections $\nabla^\prime$ given by
(\ref{Rpr}), it is valid $R^\prime = 0$.
\begin{proposition}
The complex connections $\nabla^\prime$ defined by (\ref{2-1}) and
(\ref{2-2}) are flat on $(G,J,g)$.
\end{proposition}

\bigskip

\noindent\emph{Marta Teofilova\\
Department of Geometry\\
Faculty of Mathematics and Informatics\\
University of Plovdiv\\
236 Bulgaria Blvd.\\
4003 Plovdiv, Bulgaria\\
e-mail:} \verb"marta@uni-plovdiv.bg"


\begin{thebibliography}{99}

\bibitem{Gan} G.~Ganchev, A.~Borisov, \emph{Note on the almost complex
manifolds with a Norden metric}, Compt.~Rend.~Acad.~Bulg.~Sci. 39(5)
(1986), 31--34.

\bibitem{Gan-Mih} G.~Ganchev, V.~Mihova, \emph{Canonical connection and the canonical conformal
group on an almost complex manifold with $B$-metric}, Ann.~Univ.
Sofia Fac.~Math.~Inform., 81(1) (1987), 195--206.

\bibitem{Gan-Gri-Mih2} G.~Ganchev, K.~Gribachev, V.~Mihova, $B$\emph{-connections and their conformal
invariants on conformally K\"{a}hler manifolds with $B$-metric},
Publ.~Inst.~Math. (Beograd) (N.S.) 42(56) (1987), 107--121.

\bibitem{Gri-Mek-Dje} K.~Gribachev, D.~Mekerov, G.~Djelepov, \emph{Generalized }$\emph{B}$\emph{-manifolds},
Compt.~Rend.~Acad.~Bulg.~Sci. 38(3) (1985), 299--302.

\bibitem{Ha} A.~Hayden, \emph{Subspaces of a space with torsion}, Proc. London Math. Soc. 34 (1932), 27--50.

\bibitem{Im} T.~Imai, \emph{Notes on semi-symmetric metric connections}, Tensor N.S.
24 (1972), 293--296.

\bibitem{Ko-No} S.~Kobayshi, K.~Nomizu, \emph{Foundations of differential
geometry} vol. 1, 2, Intersc.~Publ., New York, 1963, 1969.

\bibitem{N-N} A.~Newlander, L.~Nirenberg, \emph{Complex analytic coordinates
in almost complex manifolds}, Ann.~Math. 65 (1957), 391--404.

\bibitem{N} A.~P.~Norden, \emph{On a class of four-dimensional
A-spaces}, Russian Math. (Izv VUZ) 17(4) (1960), 145--157.


\bibitem{Si} S.~D.~Singh, A.~K.~Pandey, \emph{Semi-symmetric metric connections in an almost Norden metric manifold},
Acta Cienc. Indica Math. 27(1) (2001), 43--54.


\bibitem{Teo3} M.~Teofilova, \emph{Almost complex connections on almost complex manifolds
with Norden metric}, In: Trends in Differential Geometry, Complex
Analysis and Mathematical Physics, eds. K.~Sekigawa, V.~Gerdjikov
and  S.~Dimiev, World Sci.~Publ., Singapore (2009), 231--240.


\bibitem{Ya1} K.~Yano, \emph{Affine connections in an almost product space},
Kodai Math. Semin. Rep. 11(1) (1959), 1--24.

\bibitem{Ya2} K.~Yano, \emph{Differential geometry on complex and almost
complex spaces}, Pure and Applied Math. vol. 49, Pergamon Press
Book, New York, 1965.

\bibitem{Ya3} K.~Yano, \emph{On semi-symmetric metric connection},
Rev. Roumanie Math. Pure Appl. 15 (1970), 1579--1586.

\end{thebibliography}
\end{document}